\documentclass[12point]{amsart}
\usepackage[utf8]{inputenc}
\usepackage{amsmath,amssymb,mathrsfs}
\usepackage{amsthm}
\usepackage{comment}

\usepackage[pdftex]{graphicx,color}

\usepackage{concmath} 

\def\a{\alpha}

\def\f{\varphi}
\def\w{\omega}

\def\s{\sigma}

\newcommand{\C}{\mathbb{C}}

\newcommand{\R}{\mathbb{R}}
\newcommand{\Rd}{\mathbb{R}^d}
\newcommand{\Cd}{\mathbb{C}^d}

\newcommand{\Z}{\mathbb{Z}}

\newcommand{\mr}[1]{\mathrm{#1}}

\newcommand{\nrm}[1]{\left\|{#1}\right\|}
\newcommand{\set}[1]{\left\{#1\right\}\relax}
\newcommand{\scal}[1]{\left\langle\relax #1 \relax\right\rangle}

\newtheorem{thm}{Theorem}[section]
\newtheorem{defn}[thm]{Definition}
\newtheorem{prop}[thm]{Proposition}
\newtheorem{lem}[thm]{Lemma}
\newtheorem{cor}[thm]{Corollary}
\newtheorem{rem}[thm]{Remark}
\newtheorem{exa}[thm]{Example}
\newtheorem{ques}[thm]{Question}

\title{On root frames in $\Rd$}
\author{Mostafa Maslouhi and Kasso A.~Okoudjou}
\address{
	Mostafa Maslouhi
	National School of Applied Sciences (ENSAK)\\
	Ibn Tofail University. Kenitra 14000. Morocco.
}
\address{Kasso A.~Okoudjou, Department of Mathematics, Tufts University, Medford, MA 02155, USA
}
\email{mostafa.maslouhi@uit.ac.ma
}
\email{	kasso.okoudjou@tufts.edu
}

\date{\today }



\begin{document}
\maketitle

\begin{abstract} 	
\emph{A root frame} for $\mathbb{R}^d$ is a finite frame whose vectors form a root system. In this note we establish some elementary properties of this class of frames and prove that root frames constitute a subclass of scalable frames. In addition, we show that  root frames are examples of a larger class of frames called \emph{eigenframes}. 
\end{abstract}

\section{Introduction and preliminaries} 
\label{sec:introduction}
The goal of this note is to introduce and investigate the properties of a class of finite frames associated to root systems in $\R^d$. We begin with the following definition

\begin{defn}\label{def:rsys}
A subset $R \subset \Rd\setminus \{0\}$ is a root system provided that  
$$\s_\a(\beta)\in R,\quad \forall \a,\beta\in R,$$ 
where $\s_\a$ is the reflection defined by $$\s_\a(x):=x-2\scal{x,\a}\a/\nrm{\a}^2,\quad x\in\Rd.$$
In addition, $R$ is said to be a regular root system  if it does not contain a proper sub-root system.
\end{defn}

It is known that if $R$ is a root system, then   there exists $\beta\in\Rd$ such that $\scal{\alpha,\beta}\neq 0$ for all $\alpha\in R$. Therefore, the set 
$$R_{+}:=R_{+,\beta}=\set{\alpha\in R, \scal{\alpha,\beta}>0},$$ is called a positive subsystem of $R$. Note that $\#R=2\#R_{+,\beta}$ for all $\beta$. We refer to \cite{coxeter1973regular, Dunkl14, grove1996finite,humphreys1990reflection} for more details on root systems.

The class of finite frames we will investigate can now be defined as follows. 

\begin{defn}\label{def:root_frame} Let $R=\{\alpha_j\}_{j=1}^{2N}$ be a root system on $\R^d$. The collection  $\Phi_{R }=\left\{\a\right\}_{\a\in R_+}$ is called a root frame in $\mathbb{R}^d$ if $R_+$ spans $\Rd$. 
\end{defn}

\begin{rem}

\begin{enumerate}
    \item[(a)] While root systems have been the subject of many investigations, to the best of our knowledge, their spanning properties has received less attention. Nonetheless, we point out that a root frame in $\Rd$ as defined above is the same as  a rank $d$ root system defined in \cite{Dunkl14}, and as  an effective root system \cite[Proposition 4.1.2]{grove1996finite}.
    \item[(b)] Given a root system $R$, we let  $W:=W_R=\scal{\sigma_{\alpha}}$ be the (finite) group generated by  the reflections $\sigma_\alpha$ where $\alpha \in R$. $W$ is called the  Weyl (or Coxeter) group associated to $R$, and is a subgroup of the orthogonal group $O(d)$, \cite[Section 1.2]{humphreys1990reflection}. 
    \item[(c)] If $\Phi$ is a root frame for $\Rd$, then, $W_R \Phi=\Phi$. That is $W_R$ is a group of isometries leaving $\Phi$ invariant. We note that the symmetry group of (tight) frames was investigated in  \cite{ValWal2010}.
    \item[(d)] A parameter function associated to the root system $R$ is any function $k:R\to \C$ which is $W$-invariant. Given such a positive parameter function $k$ and a root system $R=\{\a_j\}_{j=1}^{2N}$, we could consider a root frame of the form $\Phi_{R, k}=\{\sqrt{k(\a)}\a\}_{\a\in R_{+}}$. As we will show in the sequel, this can be be viewed as scaling the length of each of the frame vector $\a$ and can be understood in the context of scalable frames \cite{Chenscalable, KOPT}.
\end{enumerate}
\end{rem}

In the rest of the note we establish some properties of root frames including their classification. In addition, we  prove that all root frames are scalable. Finally, we show that root frames constitute a sub-class of a family of frames we called eigenframes which can be viewed themselves as examples of  fusion frames. 

\section{Root frames and their frame operators in $\Rd$}\label{sec:root_frames}
In this section we  prove some basic properties of root frames focusing on their frame operators whose spectral properties seem quite unique compared to other frames.

\subsection{Elementary properties of root frames}\label{subsec2.1}
Let $\Phi_{R}$ be a root frame associated to the root system $R$. Without loss of generality, we assume  that $\Phi_{R}=\left\{\a\right\}_{\a\in R_+}$ is such that   the roots $\alpha$ are unit-norm vectors.  In addition, the frame operator associated to $\Phi_{R}$ will be denoted by $S_{R}$ and is given by  
\begin{equation}\label{eq:frame_op_expr}
S_{R}=  \sum_{\a\in R_+} \a \a^T
\end{equation}  where  $\a \a^T$ is the rank-one symmetric operator given by $$(\a \a^T)(x):=\langle x,\a \rangle \a, \ x \in \Rd.$$
When the root system $R$ is fixed, we will denote the corresponding frame operator by $S$ when there is no confusion. 

We next establish a number of properties of root frames. The first such result shows that the associated frame operator is independent of the choice of  the positive root. Consequently, we will also assume that for a root frame $\Phi_{R}=\left\{\a\right\}_{\a\in R_+}$, the positive roots system $R_+$ is fixed. 

\begin{prop}\label{pro:Independency_from_choice_beta}
Let $R$ be a root system and $R_{+,\beta_1}$, $R_{+,\beta_2}$ two associated positive root systems. Denote the corresponding frame operators by $S_{\beta_1}$ and $S_{\beta_2}$. Then $S_{\beta_1}=S_{\beta_2}$.
\end{prop}
\begin{proof}
We start by observing that 
$$S_{\beta_1} = \sum_{\a\in R_{+,\beta_1}}\a \a^T
	= \sum_{\a\in R_{+,\beta_1}\cap R_{+,\beta_2}} \a \a^T+\sum_{\a\in R_{+,\beta_1}\setminus R_{+,\beta_2}} \a \a^T.$$

The result follows since the map $\alpha \to -\alpha$ is a  bijection from $R_{+,\beta_1}\setminus R_{+,\beta_2}$ onto 
$R_{+,\beta_2}\setminus R_{+,\beta_1}.$
\end{proof}

Our first main result deals with the spectral decomposition of the frame operator associated to a root frame. In particular, we show that  each of its vector is an eigenvector of the corresponding frame operator. To the best of or knowledge, except for tight frames, root frames seem to be the only class of frames with this property. 

\begin{thm}\label{thm:spectralRF}
Suppose that $\Phi_{R}=\{\a\}_{\a \in R_{+}}$ is a root frame for $\Rd$ with frame operator $S$. Then each frame vector $\alpha \in R_+$  is an eigenvector for $S$ with eigenvalue $\lambda_\a$ given by

 $$\begin{cases}
S\alpha =\lambda_{\alpha}\alpha\\
\lambda_\a:=\scal{S \a,\a}=\sum_{\beta\in R_+} \scal{\a,\beta}^2.
\end{cases}$$

	Consequently, the spectrum of $S$ is  $\set{\lambda_{\alpha},\, \alpha\in R_+}$ and $R_+$ contains a basis of $\Rd$ consisting of eigenvectors of $S$.

\end{thm}

\begin{proof}
We first observe that the frame operator can be written as 
$$S= \frac{1}{2} \sum_{\beta \in R} \beta \beta^T.$$
In addition,  $S$ commutes with the action of $W_R$. That is for each $x\in \Rd$ and each $\a \in R$  $$ \s_\a (Sx)= S( \s_\a x).$$

It follows that for $\a \in R_+$
		$$\a \a^T( S)= S(\a \a^T) \iff \a (S \a)^T= (S\a) \a^T.$$ Because $S$ is symmetric  we get
		$S \a= \scal{S \a,\a}\a =\lambda_\a \a.$
\end{proof}

The following is a simple consequence of Theorem~\ref{thm:spectralRF}. 

\begin{cor}\label{co:R_i,+_is_root_syst} Suppose that $\Phi_{R}=\{\a\}_{\a \in R_{+}}$ is a root frame for $\Rd$ with corresponding frame operator $S$. Let $\{\lambda_i\}_{i=1}^r$ be the list of the distinct eigenvalues 
$\{\lambda_{\alpha}\}_{\alpha\in R_+}$ of $S$. For each $i=1, \hdots, r$ let 
$$\begin{cases}
R_{i}:=\{\a\in R_+:\,\lambda_\a =\lambda_i\}\\
E_i:=\mr{span}R_{i}
\end{cases}$$ with $d_i=\dim E_i$.

The following statements hold. 
\begin{enumerate}
\item[(a)]  $\{R_{i}\}_{i=1}^r$ is a collection of pairwise orthogonal and $W_R-$invariant sub-root systems.

\item[(b)] 	For all $i=1,\dots,r$, $E_i=\ker(S-\lambda_i)$ and $ \lambda_i\times d_i=\# R_{i, +}.$ 
\item[(c)]  If $A\leq B$ denote the optimal frame bounds of $\Phi_{R}$, then
	$$A\leq \frac{\#R_{+}}{d}\leq B.$$
\item[(d)] 	For each $i=1,\dots,r$, $\Phi_{R_i}:=\{\a\}_{\a\in R_{i,+}}$ is a tight frame for $E_i$ with $S_{R_{i,+}}=\lambda_i Id_{E_i}$. Furthermore, the root frame $\Phi_{R}=\{\a\}_{\a \in R_{+}}$ is tight if and only if the  root-system $R$ 
	is regular.
\end{enumerate}
\end{cor}

\begin{proof} Part (a) is straightforward and we omit its proof.

	\begin{enumerate}
	\item[(b)] The fact that  $\lambda_i \times d_i=\# R_{i, +}$ follows from taking the trace of the matrices in: $$\lambda_i Id_{E_i}=S_{|E_i}=\sum_{\alpha\in R_{i,+}}\alpha\otimes \alpha.$$

    \item[(c)] Given that $A,B$ are the optimal frame bounds, We see from the definition that 
	$$Ad\leq \mr{trace}(S)=\sum_{\a\in R_{+}}\nrm{\alpha}^2=\#R_{+}\leq Bd.$$

\item[(d)] The fact that $\Phi_{R_{i}}$ is a tight frame for its span $E_i$ is trivial. It follows that the root frame $\Phi_{R}$ is tight if and only if $R$ is a regular root system. 
	\end{enumerate}
\end{proof}

Another immediate consequence of the spectral properties of $S$ is the construction of Parseval frames starting with any   root system $R$. In fact, the next result shows that every root frame is scalable \cite{Chenscalable, CKLMNarayanS14, KOPT}. That is, we can always rescale each vector in a root frame to obtain a Parseval frame.

\begin{prop}\label{cor:Parseval_frame_R}
	The canonical dual  of the root frame $\Phi_{R}$ is the root frame generated by the same root system $R$ and given by
	 $$\tilde{\Phi}_{R}:=\bigg\{\sqrt{\frac{1}{\lambda_\a}}\a\bigg\}_{\a\in R_+}.$$
	
	Furthermore, $\tilde{\Phi}_{R}$ is a Parseval frame and $$\frac{1}{d}\sum_{\a\in R_+}\frac{1}{\lambda_\a}=1.$$
	
	Consequently, the root frame $\Phi_{R}$ is a scalable frame.
\end{prop}

\begin{proof}
Let $S$ be the frame operator for the the root frame $\Phi_{R}$. For $x\in \Rd$ we have 
$$Sx=\sum_{\alpha \in R_{+}}  \scal{x, \alpha}\alpha. $$
Thus $$x=\sum_{\alpha \in R_{+}}  \scal{x, \alpha}S^{-1}\alpha =\sum_{\alpha \in R_{+}} \tfrac{1}{\lambda_{\alpha}} \scal{x, \alpha}\alpha.$$

\end{proof}

\begin{rem} 
\begin{enumerate} 
\item[(a)] Because every root frame $\Phi_{R}$ is scalable, we conclude that the ellipsoid of minimum volume (also known as the L\"owner ellipsoid) that circumscribed the convex hull of $\{\pm \a\}_{\a \in R}$ is the unit ball \cite{Chenscalable}.
\item[(b)] An interesting question we have not been to settle is the characterization of scalable frames that are also root-frames. This reduces to proving that the group of generated by the reflections corresponding to the frame vectors is finite. See Theorem~\ref{thm:minimal_root_frame} for more. 
\end{enumerate}
\end{rem}

\subsection{Classification of root frames}\label{subsec3.2}
In this section we classify all the frames of unit-norm vectors $\Phi=\{\varphi_j\}_{j=1}^N \subset \Rd$ that are root frames. 

\begin{thm}\label{thm:minimal_root_frame}
Let $\Phi=\{\varphi_j\}_{j=1}^N$ be a frame for $ \R^d$ such that $\|\varphi_j\|=1$ for each $j$. Suppose that $C_{\Phi}$ is  the group generated by the reflections $\{\s_{\f_j}\}_{j=1}^N.$  Let 
$$ R_{\Phi}:=\set{g \varphi_j,\, g\in C_{\Phi}, j=1, \hdots, N}.$$
The set  $R_\Phi$ is a root frame if and only if the group $C_{\Phi}$  is finite. In this case, the initial frame $\Phi$ is contained in the  root frame $R_\Phi$. 
\end{thm}
\begin{proof} 	Suppose that $C_{\Phi}$ is a finite group. Clearly, $R_\Phi$ is a frame, since it can be written as the finite union of images of $\Phi$ under the reflections $g \in C_\Phi$. We only need to show that  $R_{\Phi}$ is a root-system in $\Rd$. Indeed,  let $\a_1=g_1\varphi_1,\a_2=g_2\varphi_2 \in R_{\Phi}$ where with $g_1, g_2\in C_{\Phi}$ and $\varphi_1, \varphi_1\in \Phi$. We have 
		$$\s_{\a_1}(\a_2)=\s_{g_1 \varphi_1}(g_2 \varphi_2)=g_2\s_{g_2^{-1}g_1\varphi_1}( \varphi_2)=h \varphi_2 \in R_\Phi$$ where $h=g_2\s_{g_2^{-1}g_1\varphi_1} \in C_{\Phi}$.

The converse is trivially proved since assuming that  $R_\Phi$ is a root frame implies that it is both a frame (hence a finite set) and a root system. 
\end{proof}

\begin{rem} 
 Recall that the spark of a frame $\Phi=\{\varphi_k\}_{k=1}^N \subset \R^d$ is the cardinality of the smallest linearly dependent subset of $\Phi$ \cite{ACM12}. If $\Phi$ is a root frame,  
then given $k\neq \ell$ there must exists $j\neq k, \ell$ such that $\s_{\varphi_{k}}(\varphi_{\ell})=\pm \varphi_j=\varphi_\ell - 2\scal{\varphi_\ell, \varphi_k}\varphi_k$. Thus $\{\varphi_j, \varphi_k, \varphi_\ell\}$ must be linearly dependent and hence the spark of the frame must be at most $3$. Consequently, if a frame $\Phi=\{\varphi_k\}_{k=1}^N \subset \R^d$ is such that every subset of three vectors is linearly independent, then the frame is not (contained in) a root frame. This is the case for any frame with spark greater or equal to $4$.  
\end{rem}

\subsection{Examples of root frames in $\mathbb{R}^d$}\label{sec:examples}

In this section we give some examples of root frames. 

\begin{exa}\label{example1}
\begin{enumerate}
    \item Let $\{e_1,\dots,e_d\}$ be an  orthonormal basis of $\R^d$. Then, $$R_+=\set{e_i-e_j,e_i+e_j, 1\leq i< j\leq d}\cup\set{e_i, i=1,\dots,d}$$ is a positive root system of rank $d$ called a type $B_{d}$ root system. 
	 
	By considering its frame operator, one can show that $R_+$ is a tight root-frame (of $d^2$ vectors) for   $\Rd$.
    \item Let $D_n$ be the dihedral group  of order 
	  $n$, $n\ge 2$. It is the group of symmetries of a regular convex polygon of $n$ vertices in the Euclidean plan $\R^2$. If we identify $z=x_1+i x_2\in \C$ with $z=(x_1,x_2)\in\R^2$ and set  $w=e^{i\pi/n}$, then the rotations in $D_{n}$ are the transformation $r_j:\, z\mapsto z w^{2j}$ and the reflections are given by $\s_j:\, z\mapsto \overline{z} w^{2j}$, $j=0,\dots, n-1$. It can be proved that  $R_+=\set{i\w^j,\, j=0,\dots,n-1}$ is a positive root-system that is also a tight root-frame for $\R^2$. Observe that this tight root frame for $\R2$ is related to the tight frames obtained by taking two rows from the $2n\times 2n$ DFT matrix.

	 \item The symmetric group $S_d$ operates on $\Rd$ by its action on the components on the canonical basis. That is, for  $x=(x_1,\dots,x_d)\in\Rd$ we have $$\s\,x=(x_{\s(1)},\dots,x_{\s(d)})\quad \forall \s\in S_d.$$
	 
	 Thus, the  transposition $(ij)$ plays the same  role as the  reflection $\s_{ij}$ defined by par $\s_{ij}(e_i-e_j)=- (e_i-e_j)$, where $\{e_1,..,e_d\}$ is the standard basis of $\Rd$. 
	 
	A positive sub root-system associated to $S_d$ is given by
	 $$
	 	R_+=\set{e_i-e_j,1\le i<j\le d}.
	 $$ One can show that  $0$ is an eigenvalue of the frame operator $S$ of $R_+$. Consequently, $R_+$ is an example of a positive root system that is not a root frame.
 
\end{enumerate}
\end{exa}

\section{Eigenframes (EF) }\label{sec:EF} We end this note with an extension of the notion of root frames. 
From Corollary~\ref{co:R_i,+_is_root_syst}, it follows that  any root frame $\Phi_R$ can be written as   $\Phi_{R}=\cup_{i=1}^r\Phi_{R_{i}}$ where $R_{\Phi_{i}}$ is a tight frame for its span, and where  $\{R_{\Phi_{i}}\}_{i=1}^r$ are mutually orthogonal. Consequently, a root frame is an example of fusion frame \cite{CaKu_FF13, casazza2008fusion}. In fact, we can introduce a class of frames for $\R^d$ (or $\C^d$) of which the root frames are examples, and which is a subclass of fusion frames. 

\begin{defn}
A frame $\Phi=\{\varphi_k\}_{k=1}^N\subset \Cd$ is called an \emph{eigenframe}(EF), provided that for each $k$, the vector $\varphi_k$ is an eigenvector of its frame operator $S:=S_\Phi=\sum_{k=1}^N \varphi_k\varphi_k^*$.
\end{defn}

Assume that $\Phi=\{\varphi_k\}_{k=1}^N$ is an EF for $\Cd$. Let $\{\lambda_k\}_{k=1}^r$ denote the set of distinct eigenvalues of $S$ with 

$$S\varphi_k=\lambda_k \varphi_k,\quad\quad  \forall k=1, \hdots, N.$$  Define 
\begin{equation}\label{eq:def_E_r}
	E_k:=\set{x\in\Rd,\ S x=\lambda_k x},
\end{equation}
and let $P_k$ denotes the orthogonal projection onto $E_k$. It is easy to see that for all $n$ we have	\begin{equation}\label{eq:project_expr}
	P_n=\frac{1}{\lambda_{n}} \sum_{\varphi_k\in E_{\lambda_{n}}}\varphi_k\varphi_k^*,
\end{equation}
where $E_{\lambda_{n}}=\Phi \cap E_{n}$.

It is straightforward to establish the following result which should be compared to Corollary~\ref{co:R_i,+_is_root_syst}. 

\begin{prop}\label{pro:gen_properties_EF} Let  $\Phi=\{\varphi_k\}_{k=1}^N$ be an EF for $\Cd$. The following statements hold.   

\begin{enumerate}
    \item[(a)] For each $n=1, 2, \hdots, r$ $$\lambda_{n}d_n= \sum_{\varphi_k\in E_{\lambda_{n}}}\nrm{\varphi_k}^2,$$ where $d_n=\dim E_n$.
    \item[(b)] For all  $u\in E_{\lambda_{n}} \cap \Phi$ we have 
 \begin{equation*}\label{eq:eig_val_bounds}
 \lambda_{n}\geq c_\Phi(u)\nrm{u}^2,
 \end{equation*}
where $c_\Phi(u)$ is the number of times $u$ appears in $\Phi$. In addition, equality holds in this inequality if and only if $\scal{u, \varphi_k}=0$ for all $u\neq \varphi_k \in \Phi$. 
\end{enumerate}
\end{prop}

The following result is an extension of Theorem~\ref{thm:spectralRF} characterizing EFs through their frame operator. We omit its simple  proof.

\begin{thm}\label{thm:charact}
	Let $\Phi=\{\varphi_k\}_{k=1}^N$ be a frame in $\Cd$. Then the following statements are equivalent
	\begin{enumerate}
		\item[(a)]\label{itm:f_EF} $\Phi$ is an EF,
		\item[(b)]\label{itm:f_commute_with_s_f} $S$ commutes with $\s_{\varphi_k}$ for all $k=1,\dots,N$,
		\item[(c)]\label{itm:f_commute_with_f_O_f} $S$ commutes with $\varphi_k\varphi_k^*=\scal{\cdot, \varphi_k}\varphi_k$ for all $k=1,\dots,N$.
		\item[(d)]\label{itm:constructing_EF} There exist mutually orthogonal subspaces $W_1$,$\dots$, $W_r$ of $\Cd$ such that $\Cd = \oplus_{i=1}^r W_i$ with $W_i =Span(\Phi_{i})$ where $\Phi_i$ is a tight frame in $W_i$ and $\Phi=\cup_{i=1}^r\Phi_i$.
	\end{enumerate}
\end{thm}

\begin{rem}\label{pro:EF_scalable}
	\begin{enumerate}
	    \item[(a)]It is easy to extend  Proposition~\ref{cor:Parseval_frame_R} and to prove that any eigenframe $\Phi$ is scalable. Indeed, if  $\Phi=\{\varphi_k\}_{k=1}^N$  is an EF,   we have $S \varphi_k=\lambda_k \varphi_k$. It follows that  
	$$Id_{\Cd}= \sum_{k=1}^N \frac{1}{\lambda_k}\varphi_k\varphi_k^T,$$ which shows that $\{\frac{1}{\sqrt{\lambda_k}}\varphi_k\}_{k=1}^N$ is Parseval frame.
	\item[(b)] If $\Phi=\{\varphi_k\}_{k=1}^N$ is an EF, then its Gram matrix is block diagonal. 
		
	\end{enumerate}
\end{rem}

\section*{Acknowledgment}
M.Maslouhi would like to thank the Moroccan-American Commission For Educational \& Cultural Exchange (MACECE) for their grant and Tufts university for their support during his Fulbright project.

K.A.O. was  partially supported by a grant from   the National Science Foundation grant DMS 1814253.

\bibliographystyle{abbrv}
\bibliography{root-frames}

\end{document}